\documentclass[12pt]{amsart}

\usepackage{custom_preamble}

\begin{document}

\title{The coset and stability rings}

\author{\tsname}
\address{\tsaddress}
\email{\tsemail}

\begin{abstract} 
We show that if $G$ is a discrete Abelian group and $A \subset G$ has $\|1_A\|_{B(G)} \leq M$ then $A$ is $O(\exp(\pi M))$-stable in the sense of Terry and Wolf.
\end{abstract}

\maketitle

\setcounter{section}{1}

In \cite{terwol::0} Terry and Wolf, inspired by ideas in model theory, introduce the notion of stability for sets in Abelian groups, and very quickly after there followed a number of papers building on their work \emph{e.g.} \cite{conpilter::0,alofoxzha::0,sis::1} and \cite{terwol::1}.  In this note we develop a relationship between stability and the Fourier algebra.

Suppose that $G$ is a (possibly infinite) Abelian group.  Following \cite[Definition 1]{terwol::1}, for $k \in \N$ we say $A \subset G$ has the \textbf{$k$-order property} if there are vectors $a,b \in G^k$ such that $a_i +b_j \in A$ if and only if $i \leq j$.  If $A$ does not have the $k$-order property it is said to be \textbf{$k$-stable}.  Note that the order property is monotonic so if $A$ has the $k$-order property then it has the $(k-1)$-order property (for $k \geq 2$), and mutatis mutandis for stability.

Write $\mathcal{S}_l(G)$ for the set of subsets of $G$ that are $l$-stable and $\mathcal{S}(G)$ for their union over all $l \in \N$.  We begin with some examples from \cite{terwol::0}:
\begin{lemma}[The empty set and cosets]\label{lem.cs}
$\mathcal{S}_1(G)=\{\emptyset\}$ and $\mathcal{S}_2(G)=\mathcal{S}_1(G)\cup \bigcup_{H \leq G}{G/H}$.
\end{lemma}
\begin{proof}
The first equality is immediate.  For the second, from \cite[Example 1 \& Lemma 2]{terwol::0} (or Lemma \ref{ex.duc} later) we have $\bigcup_{H \leq G}{G/H} \subset \mathcal{S}_2(G)$.  And conversely if $A$ is $2$-stable and $x,y,z \in A$ then putting $a_1=x$, $a_2=y$, $b_1=z-x$ and $b_2=0_G$ we see that $a_1+b_1,a_1+b_2,a_2+b_2 \in A$ by design.  Since $A$ does \emph{not} have the $2$-order property it follows that $y+z-x=a_2+b_1 \in A$, and so $A+A-A \subset A$ and if $A$ is non-empty it follows (by \emph{e.g.} \cite[\S3.7.1]{rud::1}) that $A$ is a coset of a subgroup.  The result is proved.
\end{proof}
More interesting than the examples, Terry and Wolf show that $\mathcal{S}(G)$ has a ring structure.  Recall that $\mathcal{R}$ is a \textbf{ring of subsets of $G$} if $\mathcal{R} \subset \mathcal{P}(G)$ is closed under complements and finite intersections (and hence finite unions).  The prototypical example is $\mathcal{P}(G)$ itself; \cite[Lemmas 1 \& 2]{terwol::0} give the following.
\begin{theorem}[Terry--Wolf Stability Ring]\label{thm.twstr}
Suppose that $G$ is an Abelian group.  Then $\mathcal{S}(G)$ is a translation-invariant ring of subsets of $G$.
\end{theorem}
It may help to compare this with \emph{e.g.} \cite[Exercise 8.2.9]{tenzie::}, the folklore fact that the set of stable formulas is closed under boolean combinations.

We write $\mathcal{W}(G)$ for the \textbf{coset ring} of $G$, that is the minimal translation-invariant ring of sets containing all cosets of subgroups of $G$.  The coset ring has received attention in harmonic analysis (see \cite[Chapters 3 and 4]{rud::1}), and in view of Lemma \ref{lem.cs} and Theorem \ref{thm.twstr} we have $\mathcal{W}(G) \subset \mathcal{S}(G)$; it is natural to ask whether we have equality.

For any $A \subset G$ we have $A=\bigcup_{x\in A}{(x+\{0_G\})}$ and so if $G$ is finite any set is a finite union of cosets of subgroups of $G$ and in particular $\mathcal{W}(G)=\mathcal{S}(G)$, but if $G$ is not finite then things may be different.  To see this we need a new example of sets of low stability.

Following\footnote{Deviating from other definitions \emph{e.g.} \cite[Definition 4.26]{taovu::} if $G$ has $2$-torsion.} \cite[Definition 1]{cil::0} we say a set $A \subset G$ is a \textbf{Sidon set} (also known as a $B_2$-set) if whenever $x-y=z-w$ for some $(x,y,z,w) \in A^4$ we have $x=y$ or $x=z$.  
\begin{lemma}[Sidon sets are $3$-stable]
Suppose that $G$ is an Abelian group and $A \subset G$ is a Sidon set.  Then $A$ is $3$-stable.
\end{lemma}
\begin{proof}
Suppose that $a,b \in G^3$ witness the $3$-order property in $A$.  Then $a_i+b_j \in A$ whenever $i \leq j$, and so $(a_1+b_2, a_1+b_3, a_2+b_2, a_2+b_3) \in A^4$.  But then
\begin{equation*}
(a_1+b_2)-(a_1+b_3)=b_2-b_3 = (a_2+b_2)-(a_2+b_3),
\end{equation*}
so by Sidonicity either $a_1+b_2=a_1+b_3$ and we have $b_2=b_3$; or $a_1+b_2=a_2+b_2$ and we have $a_1=a_2$.  In the former case we have $a_3+b_2 = a_3+b_3 \in A$ -- a contradiction.  In the latter we have that $a_2+b_1=a_1+b_1 \in A$ -- a contradiction.  It follows that $A$ is $3$-stable.
\end{proof}
On the other hand there are (at least if $|G|>3$) sets in $\mathcal{S}_3(G) \setminus \mathcal{S}_2(G)$ that are not Sidon, for example a subgroup of size at least $4$ with the identity removed.\footnote{Such a set is the intersection of a subgroup with the complement of a subgroup.  The former is $2$-stable (as recorded in Lemma \ref{lem.cs}); the latter is $3$-stable since subgroups are $2$-stable and the complements of $k$-stable sets are $(k+1)$-stable by \cite[Lemma 1]{terwol::0}.  It follows by \cite[Lemma 3]{terwol::0} that the resulting intersection is $3$-stable.  On the other hand it is not a coset of a subgroup, and so not $2$-stable, and not Sidon since if $A$ is a Sidon subset of a group $H$ then $|A|^2-|A|+1 \leq |H|$, but $(|H|-1)^2 - (|H|-1)+1 > |H|$.}

The set $\mathcal{T}:=\{1,3,9,27,\dots\}$ is an example of an infinite Sidon set in the integers.  While stability need not be preserved by passing to subsets, Sidonicity is and so every subset of $\mathcal{T}$ is also Sidon and \emph{a fortiori} $3$-stable, so $\mathcal{S}(\Z)$ is uncountable.  On the other hand there are countably many cosets of subgroups of $\Z$ and so $\mathcal{W}(\Z)$ is countable, and we conclude that $\mathcal{S}(\Z)\neq \mathcal{W}(\Z)$.

In view of the above discussion it is tempting to ask for families of sets in $\mathcal{S}_{k+1}(G)$ that are not in the ring generated by $\mathcal{S}_k(G)$ for $k>2$ -- the `irreducible' elements of $\mathcal{S}_{k+1}(G)$.

Cohen's idempotent theorem \cite[\S3.1.3]{rud::1} tells us that $\mathcal{W}(G)$ is equal to the Fourier algebra $\mathcal{A}(G)$.  To define the latter we take $G$ to be discrete and write $\wh{G}$ for the compact Abelian dual group of homomorphisms $G \rightarrow S^1:=\{z \in \C: |z|=1\}$, and if $\mu \in M(\wh{G})$ put
\begin{equation*}
\wh{\mu}(x):=\int{\lambda(x)d\mu(\lambda)} \text{ for all }x \in G.
\end{equation*}
For $f:G \rightarrow \C$, if there is some $\mu \in M(\wh{G})$ such that $f=\wh{\mu}$ then $\mu$ is unique \cite[\S1.3.6]{rud::1} and we put
\begin{equation*}
\|f\|_{B(G)}:=\|\mu\|:=\int{d|\mu|}.
\end{equation*}
With this we write
\begin{equation*}
\mathcal{A}(G):=\left\{A \subset G: \|1_A\|_{B(G)} < \infty\right\},
\end{equation*}
 which it turns out is a translation-invariant \cite[\S1.3.3 (c)]{rud::1} ring of sets \cite[\S3.1.2]{rud::1}.

Since $\mathcal{A}(G)=\mathcal{W}(G) \subset \mathcal{S}(G)$, we see that if $\|1_A\|_{B(G)} < \infty$ then $A \in \mathcal{S}(G)$, and it is natural to wonder if there is a universal monotonic $F:\R_{\geq 1} \rightarrow \N$ such that
\begin{equation}\label{eqn.mono}
\|1_A\|_{B(G)} \leq M \Rightarrow A \in \mathcal{S}_{F(M)}(G).
\end{equation}
There is fairly direct approach via a quantitative version of Cohen's theorem.  To describe this we make a definition: given $H \leq G$ and $\mathcal{S} \subset G/H$ we write\footnote{For $S \subset G$ we write $\neg S:=G \setminus S$.} $\mathcal{S}^*:=\mathcal{S} \cup \left\{\neg \bigcup{\mathcal{S}}\right\}$, that is the partition of $G$ into cells from $\mathcal{S}$ and an additional cell that is everything else.  We say that $A$ has a \textbf{$(k,s)$-representation} if there are subgroups $H_1,\dots,H_k \leq G$, and sets $\mathcal{S}_1\subset G/H_1,\dots,\mathcal{S}_k\subset G/H_k$ of size at most $s$ such that $A$ is the (disjoint) union of some cells in the partition\footnote{Recall that if $\mathcal{P}$ and $\mathcal{Q}$ are partitions of the same set then $\mathcal{P} \wedge \mathcal{Q}:=\{P \cap Q: P \in \mathcal{P}, Q \in \mathcal{Q}\}$.} $\mathcal{S}_1^* \wedge \cdots \wedge \mathcal{S}_k^*$.

\begin{theorem}[Quantitative idempotent theorem, {\cite[Theorem 1.2]{gresan::0}}]\label{thm.quantcoh}
Suppose that $\|1_A\|_{B(G)} \leq M$.  Then $A$ has a $(k,s)$-representation where
\begin{equation*}
k \leq M+O(1) \text{ and }s \leq \exp(\exp(O(M^4))).
\end{equation*}
\end{theorem}
The arguments of Terry and Wolf are also quantitative, and we record some of them in a slightly stronger form than they state.  We begin with a slight extension of \cite[Example 1]{terwol::0}.
\begin{lemma}[Unions of cosets]\label{ex.duc}
Suppose that $H \leq G$ and $\mathcal{S} \subset G/H$ has size $s$.  Then $\bigcup{\mathcal{S}}$ is $(s+1)$-stable.
\end{lemma}
\begin{proof}
Suppose that $a,b \in G^{s+1}$ witness the $(s+1)$-order property in $\bigcup{\mathcal{S}}$.  By the pigeonhole principle there is some $1 \leq i < j \leq s+1$ such that $a_1+b_i$ and $a_1+b_j$ are in the same coset of $H$, whence $b_i+H=b_j+H$.  Since $a_j+b_j \in \bigcup{\mathcal{S}}$ we see that $a_j+b_i \in \bigcup{\mathcal{S}}+H=\bigcup{\mathcal{S}}$, a contradiction since $j>i$.
\end{proof}
In general the above lemma is best-possible as the next lemma shows when $G=\Z$, $H=\{0\}$, and $\mathcal{S}$ is the set of size-one subsets of an arithmetic progression.
\begin{lemma}[Arithmetic progressions]\label{lem.ap} Suppose that $A$ is an arithmetic progression of integers of size $r$.  Then $A$ has the $r$-order property.
\end{lemma}
\begin{proof}
Write $A=\{x,x+d,\dots,x+(r-1)d\}$, and let $s_i:=x-id$ and $t_i=id$ for $1 \leq i \leq r$.  Then $s_i+t_j = x+(j-i)d \in A$ if and only if $i \leq j$, and so the vectors $s,t \in \Z^r$ so defined witness the $r$-order property in $A$.  (\emph{c.f.} \cite[Lemma 6.3]{sis::1}.)
\end{proof}
Quantitatively \cite[Lemma 1]{terwol::0} is about as good as one could hope -- it says if $A$ is $s$-stable then $\neg A$ is $(s+1)$-stable -- however we shall combine it with a multi-set version of \cite[Lemma 2]{terwol::0}.

Write $r(k_1,\dots,k_m)$ for the smallest natural number such that in any $m$ colouring of the complete graph on $r(k_1,\dots,k_m)$ vertices there is some $1 \leq q \leq m$ such that the $q$th colour class contains a complete graph on $k_{q}$ vertices.
\begin{lemma}\label{lem.multicolour}
Suppose that $\bigcup_{q=1}^m{A_q}$ has the $r(k_1+1,\dots,k_m+1)+1$-order property.  Then there is some $1 \leq q \leq m$ such that $A_{q}$ has the $k_{q}$-order property.
\end{lemma}
\begin{proof}
Write $N:=r(k_1+1,\dots,k_m+1)$.  Since $\bigcup_{q=1}^m{A_q}$ has the $(N+1)$-order property there are vectors $a,b \in G^{N+1}$ so that we can colour the vertices of the complete graph on $\{1,\dots,N\}$ by giving the edge $ij$ (for $1 \leq i < j \leq N$) the colour of the smallest $q$ such that $a_{i+1} + b_{j} \in A_q$ -- this is an $m$-colouring of the complete graph on $N$ vertices.

By definition there is some $1 \leq q \leq m$ and a sequence $1 \leq s_1<\dots<s_{k_{q}+1} \leq N$ with $a_{s_i+1} + b_{s_j} \in A_{q}$ for all $1\leq i <j \leq k_{q}+1$.  On the other hand whenever $N \geq i\geq j \geq 1$ we have $a_{i+1} + b_j \not \in \bigcup_n{A_n}$ by the $N$-order property of $\bigcup_n{A_n}$, and hence $a_{s_i+1} + b_{s_j} \not \in A_{q}$ whenever $k_{q}+1 \geq i\geq j\geq 1$.  Finally let $a_i':=a_{s_i+1}$ and $b_i':=b_{s_{i+1}}$ for all $1 \leq i \leq k_{q}$, and note that $a_i'+b_j' \in A_{q}$ if and only if $i \leq j$ as required.
\end{proof}
We now use this lemma to compute an upper bound on the stability of a set in the coset ring based on the complexity of its representation.  To do this we need a bound on the multicolour Ramsey numbers.  The usual Erd{\H o}s-Szekeres argument gives (see \emph{e.g.} \cite[Corollary 3]{gregle::0}) that
\begin{equation}\label{eqn.ramsey}
r(k_1+1,\dots,k_m+1) \leq \frac{(k_1+\cdots + k_m)!}{k_1!\cdots k_m!}.
\end{equation}
Suppose now that $A$ has a $(k,s)$-representation as described before Theorem \ref{thm.quantcoh}.  Then each of the sets $\neg \bigcup{\mathcal{S}_i}$ is $s+2$ stable by Lemma \ref{ex.duc} and \cite[Lemma 1]{terwol::0}, so each cell in the partition $\mathcal{S}_1^* \wedge \cdots \wedge \mathcal{S}_k^*$ is an intersection of $k$ sets of stability at most $s+2$.  It follows from Lemma \ref{lem.multicolour} and (\ref{eqn.ramsey}) that $A$ is $t$-stable where
\begin{align*}
t & \leq r(\overbrace{r(\underbrace{s+3,\dots,s+3}_{k\text{ times}})+2,\dots,r(s+3,\dots,s+3)+2}^{(s+1)^k\text{ times}}) +2\\
& \leq r(\overbrace{k!^{s+2}+2,\dots,k!^{s+2}+2}^{(s+1)^k{\text{ times}}})+2 \leq 2^{k^{7ks}}.
\end{align*}
Plugging Theorem \ref{thm.quantcoh} into this shows that one may take $F$ in (\ref{eqn.mono}) with
\begin{equation}\label{eqn.pos}
F(M)=\exp(\exp(\exp(\exp(O(M^4))))).
\end{equation}
On the other hand, in some situations we can do far better: if $A$ is finite and $G$ is torsion-free then McGehee, Pigno and Smith's solution to Littlewood's conjecture\footnote{This theorem extends to $\wh{G}$ connected as noted in \cite[\S3, Remark (i)]{mcgpigsmi::}, and $\wh{G}$ is connected since $G$ is torsion-free by \cite[Theorem 2.5.6(c) \& Theorem 1.7.2]{rud::1}).} \cite[Theorem 2]{mcgpigsmi::} applies to show that if $\|1_A\|_{B(G)} \leq M$ then $|A|=\exp(O(M))$.  It follows from Lemma \ref{ex.duc} that $A$ is $\exp(O(M))$-stable.  This is far better than the bound in (\ref{eqn.pos}) and it is the main purpose of this note to prove a bound of this strength directly and in full generality:
\begin{theorem}\label{thm.l1stab} 
Suppose that $G$ is a (discrete) Abelian group and $A\subset G$ has $\|1_A\|_{B(G)} \leq M$.  Then $A$ is $(c_0\exp(\pi M)+1)$-stable where $c_0:=2^{-4}\exp(-\gamma)\pi =0.110\dots$ and $\gamma$ is the Euler--Mascheroni constant.
\end{theorem}
\begin{proof}
Suppose that $a,b \in G^k$ witness the $k$-order property and consider
\begin{equation*}
P:\ell_2^k \rightarrow \ell_2^k;  v \mapsto \left(\sum_{m=1}^k{1_A(a_l+b_m)v_m}\right)_{l=1}^k,
\end{equation*}
where $\ell_2^k$ denotes $k$-dimensional complex Hilbert space.

We compute the trace norm of $P$ in two ways: one showing it is large by direct calculation as it is just the trace norm\footnote{The trace norm of the adjacency matrix of a graph is sometimes called the \textbf{graph energy} \cite{gut::}.} of (a variant of) the adjacency matrix of the half-graph; on the other hand it is small as a result of the hereditary smallness of the algebra norm.

Since $a$ and $b$ witness the order property, writing $Q$ for the matrix of $P$ with respect to the standard basis we have $Q_{ij}=1$ if $i \leq j$ and $0$ otherwise.  It happens to be easier to deal with $Q^{-1}$; for reference (which can be easily checked)
\begin{equation*}
Q=\left(\begin{array}{cccc}1 & 0  & \cdots & 0\\ \vdots & \ddots & \ddots & \vdots\\ \vdots & & \ddots & 0\\ 1 & \cdots & \cdots & 1\end{array}\right)\text{ and }Q^{-1}=\left(\begin{array}{ccccc}1 & 0 & \cdots & \cdots & 0\\ -1 & 1 &\ddots  && \vdots\\  0 & -1 & 1 & \ddots& \vdots \\ \vdots & \ddots & \ddots & \ddots & 0\\ 0 & \cdots & 0 & -1 & 1\end{array}\right).
\end{equation*}
It follows that
\begin{equation*}
Q^{-1}(Q^{-1})^t=\left(\begin{array}{ccccc}1 & -1 &0 & \cdots & 0\\ -1 & 2 &-1 & \ddots  & \vdots \\  0 & -1 & 2 & \ddots& 0 \\ \vdots &\ddots  & \ddots & \ddots &-1\\ 0 & \cdots & 0 & -1 & 2\end{array}\right).
\end{equation*}
Let\footnote{Similar spectral computations to those here may be found in \emph{e.g.} \cite[\S1.4.4]{brohae::}, though we followed \cite{elk::1}.} $\omega:=\exp\left(\frac{2\pi ij}{2k+1}\right)$ and $v:=(\omega,\dots,\omega^{k})$ so
\begin{equation*}
Q^{-1}(Q^{-1})^t(v+\overline{v})^t = (2-\omega-\omega^{-1})(v+\overline{v})^t.
\end{equation*}
Of course $2-\omega-\omega^{-1}=4\cos^2\left(\frac{\pi j}{2k+1}\right)$ which takes $k$ distinct values as $1 \leq j \leq k$.  It follows that the eigenvalues of $P^{-1}(P^{-1})^\ast$ are exactly the numbers $4\cos^2\left(\frac{\pi j}{2k+1}\right)$ for $1 \leq j \leq k$, and hence the eigenvalues of $(P^{-1}(P^{-1})^\ast)^{-1}=P^*P$ are the reciprocals of these.  These reciprocals, $1/4\cos^2\left(\frac{\pi j}{2k+1}\right)$ for $1 \leq j \leq k$, are themselves distinct and so have corresponding unit eigenvectors $v^{(1)},\dots,v^{(k)}$ (of $P^*P$) which are mutually perpendicular, as are the unit vectors $w^{(1)},\dots,w^{(k)}$ defined through
\begin{equation*}
Pv^{(j)} = \frac{1}{2\cos\left(\frac{\pi j}{2k+1}\right)}w^{(j)} \text{ for } 1 \leq j \leq k.
\end{equation*}
Since $\|1_A\|_{B(G)} \leq M$ there is some $\mu \in M(\wh{G})$ with $\|\mu\|\leq M$ and $1_A=\wh{\mu}$ so, in particular,
\begin{equation}\label{eqn.above}
1_A(a_l+b_m) = \int{\lambda(a_l)\lambda(b_m)d\mu(\lambda)} \text{ for all }1\leq l,m \leq k.
\end{equation}
Let $\lambda^a,\lambda^b \in \ell_2^k$ be defined by $\lambda^a_l:=\lambda(a_l)$ and $\lambda^b_m:=\lambda(-b_m)$ for $1 \leq l,m \leq k$ so that $\|\lambda^a\|_{\ell_2^k}=\|\lambda^b\|_{\ell_2^k}=\sqrt{k}$.  Then by (\ref{eqn.above}) and linearity we have
\begin{align*}
\langle w^{(j)},Pv^{(j)}\rangle_{\ell_2^k} & = \sum_{l=1}{w^{(j)}_l\overline{\sum_{m=1}^k{1_A(a_l+b_m)v^{(j)}_m}}}\\
& = \int{\left(\sum_{l=1}{w^{(j)}_l\lambda(-a_l)}\right)\overline{\left(\sum_{m=1}^k{\lambda(b_m)v^{(j)}_m}\right)d\mu(\lambda)}}\\
& =\int{\langle w^{(j)},\lambda^a\rangle_{\ell_2^k}\overline{\langle \lambda^b,v^{(j)}\rangle_{\ell_2^k}d\mu(\lambda)}} = \int{\langle w^{(j)},\lambda^a\rangle_{\ell_2^k}\langle v^{(j)},\lambda^b\rangle_{\ell_2^k} \overline{d\mu(\lambda)}},
\end{align*}
and hence (noting the left hand side is real so that the inequality makes sense)
\begin{align}\label{eqn.s}
\sum_{j=1}^k{\langle w^{(j)},Pv^{(j)}\rangle_{\ell_2^k}} & \leq \int{\sum_{j=1}^k{\left|\langle w^{(j)},\lambda^a\rangle_{\ell_2^k}\langle v^{(j)},\lambda^b\rangle_{\ell_2^k}\right|}d|\mu|(\lambda)}\\
\nonumber & \leq \int{\left(\sum_{j=1}^k{|\langle w^{(j)},\lambda^a\rangle_{\ell_2^k}|^2}\right)^{\frac{1}{2}}\left(\sum_{j=1}^k{|\langle v^{(j)},\lambda^b\rangle_{\ell_2^k}|^2}\right)^{\frac{1}{2}}d|\mu|(\lambda)}\\ & \nonumber = \int{\|\lambda^a\|_{\ell_2^k}\|\lambda^b\|_{\ell_2^k}d|\mu|(\lambda)} \leq kM.
\end{align}
This last equality is Parseval's identity (or the generalised Pythagorean theorem) applied with the two orthonormal bases $(w^{(j)})_{j=1}^k$ and $(v^{(j)})_{j=1}^k$ and the vectors $\lambda^a$ and $\lambda^b$ respectively.

In the other direction we have
\begin{align*}
\sum_{j=1}^k{\langle w^{(j)},Pv^{(j)}\rangle_{\ell_2^k}} = \sum_{j=1}^k{\frac{1}{2\left|\cos \frac{j\pi}{2k+1}\right|}} & = \frac{1}{2}\sum_{l=0}^{k-1}{\frac{1}{\sin \frac{(2l+1)\pi}{2(2k+1)}}}\\
&= \frac{2k+1}{\pi }\sum_{l=0}^{k-1}{\frac{1}{2l+1}}\\
& \qquad \qquad + \frac{1}{2}\sum_{l=0}^{k-1}{\left(\csc\left(\frac{\pi}{2}\cdot \frac{2l+1}{2k+1}\right) - \frac{2}{\pi \cdot \frac{2l+1}{2k+1}}\right)}.
\end{align*}
The function $x \mapsto \csc\left(\frac{\pi}{2}x\right) - \frac{2}{\pi x}$ is an increasing function on $\left[-\frac{\pi}{2},\frac{\pi}{2}\right]$ and so we can apply a standard integral estimate (the details of which we omit) to see that
\begin{equation*}
\frac{1}{2}\sum_{l=0}^{k-1}{\left(\csc\left(\frac{\pi}{2}\cdot \frac{2l+1}{2k+1}\right) - \frac{2}{\pi \cdot \frac{2l+1}{2k+1}}\right)} \geq \frac{2k+1}{4}\left(\frac{2}{\pi} \log \frac{4}{\pi} - \frac{3}{\pi(2k+1)}\right).
\end{equation*}
Again, omitting details, one can use inequalities of T{\'o}th \cite[Problem E3432(i)]{padmenwol::} to estimate the harmonic numbers (the $n$th of which we denote by $H_n$) and get
\begin{align*}
\frac{2k+1}{\pi }\sum_{l=0}^{k-1}{\frac{1}{2l+1}}& = \frac{2k+1}{\pi}\left(H_{2k}-\frac{1}{2}H_k\right) \geq \frac{2k+1}{2\pi}\left(\log k + \log 4 + \gamma \right).
\end{align*}
In view of these two calculations we conclude that
\begin{align*}
\sum_{j=1}^k{\langle w^{(j)},Pv^{(j)}\rangle_{\ell_2^k}} & \geq \frac{k}{\pi}\left(\log k c_0^{-1} - \frac{1}{k}\right).
\end{align*}
Finally, $\exp(-x) \geq 1-x$ and so the above along with (\ref{eqn.s}) rearranges to give the result.

\end{proof}
In the other direction we have the examples afforded by intervals.
\begin{example}\label{ex.long} Suppose that $k \in \N$ and $G:=\Z$.  Then there is a set $A\subset G$ such that $A$ is at best $(k+1)$-stable (meaning $A$ has the $k$-order property) and writing $M:=\|1_A\|_{B(G)}$ has
\begin{equation*}
k+1 \geq c_1\exp\left(\frac{\pi^2}{4}M\right) \text{ where}\footnote{So $\frac{c_1}{c_0}= \frac{4}{\pi}\prod_{m=1}^\infty{m^{-\frac{2}{4m^2-1}}} =2\prod_{m=1}^\infty{\frac{4m^2-1}{4m^{2\cdot \frac{4m^2}{4m^2-1}}}}= 0.789\dots$, where $c_0$ is the constant in Theorem \ref{thm.l1stab}.}\text{ } c_1:=2^{-2}\exp(-\gamma)\prod_{m=1}^\infty{m^{-\frac{2}{4m^2-1}}}=0.087\dots.
\end{equation*}
\end{example}
\begin{proof}
Put $A:=\{1,\dots,k\}$.  A short calculation shows that
\begin{equation*}
M=\int_0^1{\left|\sum_{n=1}^k{\exp(2\pi i n\theta)}\right|d\theta} = \int_0^1{\frac{|\sin (\pi k\theta)|}{\sin \pi \theta}d\theta}.
\end{equation*}
Szeg{\H o} in \cite[(R)]{sze::4} gives a beautiful evaluation of this quantity (in fact the cited formula is for $k$ odd, but the same argument works for any $k$ as noted in \cite[Remark 2, \S3]{sze::4}):
\begin{align*}
\int_0^1{\frac{|\sin (\pi k\theta)|}{\sin \pi \theta}d\theta}&=\frac{16}{\pi^2}\sum_{m=1}^\infty{\frac{H_{2mk}-\frac{1}{2}H_{mk}}{4m^2-1}}\\
& < \frac{8}{\pi^2}\sum_{m=1}^\infty{\frac{\log mk +\gamma + \log 4 + \frac{1}{16(mk)^2}}{4m^2-1}}\\
& \leq \frac{4}{\pi^2}(\log k+\gamma+\log 4) + \frac{8}{\pi^2}\sum_{m=1}^\infty{\frac{\log m}{4m^2-1}} + \frac{1}{4\pi^2k^2}
\end{align*}
where the first inequality follows from the inequalities of T{\'o}th \cite[Problem E3432(i)]{padmenwol::}.  By Lemma \ref{lem.ap}, $A$ has the $k$-order property and the result is proved.
\end{proof}
Thus certainly the exponent $\pi$ in Theorem \ref{thm.l1stab} cannot be improved past $\frac{\pi^2}{4}$.  That being said, the fact that these two numbers are close leads one to wonders if the proof of Theorem \ref{thm.l1stab} above is amenable to improvement by direct analysis in the case that we are close to equality in the inequalities used.

As far as we know a better result than Theorem \ref{thm.l1stab} may be true in the model setting $G=\F_2^n$ where there are no large arithmetic progressions -- the set used in Example \ref{ex.long}.

Since $\mathcal{S}(\Z) \neq \mathcal{A}(\Z)$ we know that there is no converse to Theorem \ref{thm.l1stab}.  The following lemma (\emph{c.f.} \cite{fab::0}) shows that this is so in essentially the worst possible way.  
\begin{example}
Suppose $q$ is a prime power.  Then for $G:=\Z/(q^2+q+1)\Z$ there is a Sidon (and \emph{a fortiori} $3$-stable) set $A\subset G $ of size $q+1$ such that $\|1_A\|_{B(G)} \geq \sqrt{|A|-1+o_{|A| \rightarrow\infty}(1)}$.
\end{example}
\begin{proof}
The perfect difference set construction of Singer \cite[p381]{sin::} gives a set $A$ of size $q+1$ in $G:=\Z/(q^2+q+1)\Z$ that is a Sidon set and a direct calculation shows that $1_A = \wh{\mu}$ for $\mu \in M(\wh{G})$ with
\begin{equation*}
\int{1_\Gamma d|\mu|} = \frac{\sqrt{q}}{q^2+q+1} |\Gamma \setminus \{0_{\wh{G}}\}| + \frac{q+1}{q^2+q+1}|\Gamma \cap \{0_{\wh{G}}\}|.
\end{equation*}
It follows that
\begin{equation*}
\|1_A\|_{B(G)}=\frac{q+1}{q^2+q+1} + \frac{q^2+q}{q^2+q+1}\sqrt{q}=\sqrt{q+o_{q \rightarrow \infty}(1)},
\end{equation*}
as claimed.
\end{proof}
Note if $A$ is $2$-stable then it is a coset of a subgroup (or empty) and so $\|1_A\|_{B(G)} \leq 1$.  On the other hand a careful accounting of the constants in \cite[(3.3)]{bou::7} shows that any finite $A \subset G$ has $\|1_A\|_{B(G)} \leq \sqrt{|A|-1+o_{|A|\rightarrow \infty}(1)}$ which matches our bound above up to the little-$o$ term.  In other words $3$-stable sets can have algebra norm essentially as large as possible for their size.

\section*{Acknowledgments}

My thanks to Gabriel Conant and Julia Wolf for helpful conversations on this topic, and to an anonymous referee for a very helpful report significantly improving the paper.

\bibliographystyle{halpha}

\bibliography{references}

\newcommand{\etalchar}[1]{$^{#1}$}
\begin{thebibliography}{PMW{\etalchar{+}}91}
\expandafter\ifx\csname url\endcsname\relax
  \def\url#1{\texttt{#1}}\fi
\expandafter\ifx\csname doi\endcsname\relax
  \def\doi#1{\burlalt{doi:#1}{http://dx.doi.org/#1}}\fi
\expandafter\ifx\csname urlprefix\endcsname\relax\def\urlprefix{URL }\fi
\expandafter\ifx\csname href\endcsname\relax
  \def\href#1#2{#2}\fi
\expandafter\ifx\csname burlalt\endcsname\relax
  \def\burlalt#1#2{\href{#2}{#1}}\fi

\bibitem[AFZ19]{alofoxzha::0}
N.~{Alon}, J.~{Fox}, and Y.~{Zhao}.
\newblock {Efficient arithmetic regularity and removal lemmas for induced
  bipartite patterns}.
\newblock {\em Discrete Anal.}, (3), 2019,
  \burlalt{arXiv:1801.04675}{http://arxiv.org/abs/arXiv:1801.04675}.
\newblock \doi{10.19086/da.7757}.

\bibitem[BH12]{brohae::}
A.~E. Brouwer and W.~H. Haemers.
\newblock {\em Spectra of graphs}.
\newblock Universitext. Springer, New York, 2012.
\newblock \doi{10.1007/978-1-4614-1939-6}.

\bibitem[Bou93]{bou::7}
J.~Bourgain.
\newblock On the spectral type of {O}rnstein's class one transformations.
\newblock {\em Israel J. Math.}, 84(1-2):53--63, 1993.
\newblock \doi{10.1007/BF02761690}.

\bibitem[Cil12]{cil::0}
J.~Cilleruelo.
\newblock Combinatorial problems in finite fields and {S}idon sets.
\newblock {\em Combinatorica}, 32(5):497--511, 2012,
  \burlalt{arXiv:1003.3576}{http://arxiv.org/abs/arXiv:1003.3576}.
\newblock \doi{10.1007/s00493-012-2819-4}.

\bibitem[CPT17]{conpilter::0}
G.~{Conant}, A.~{Pillay}, and C.~{Terry}.
\newblock {A group version of stable regularity}.
\newblock {\em ArXiv e-prints}, October 2017,
  \burlalt{arXiv:1710.06309}{http://arxiv.org/abs/arXiv:1710.06309}.

\bibitem[Elk11]{elk::1}
N.~D. Elkies.
\newblock 2-norm of the upper triangular ``all-ones" matrix.
\newblock MathOverflow, August 2011.
\newblock \urlprefix\url{https://mathoverflow.net/q/72383}.

\bibitem[Fab93]{fab::0}
J.~Fabrykowski.
\newblock A note on sums of powers of complex numbers.
\newblock {\em Acta Mathematica Hungarica}, 62(3):209--210, Sep 1993.
\newblock \doi{10.1007/BF01874643}.

\bibitem[GG55]{gregle::0}
R.~E. {Greenwood} and A.~M. {Gleason}.
\newblock {Combinatorial relations and chromatic graphs.}
\newblock {\em {Can. J. Math.}}, 7:1--7, 1955.
\newblock \doi{10.4153/CJM-1955-001-4}.

\bibitem[GS08]{gresan::0}
B.~J. Green and T.~Sanders.
\newblock A quantitative version of the idempotent theorem in harmonic
  analysis.
\newblock {\em Ann. of Math. (2)}, 168(3):1025--1054, 2008,
  \burlalt{arXiv:math/0611286}{http://arxiv.org/abs/arXiv:math/0611286}.
\newblock \doi{10.4007/annals.2008.168.1025}.

\bibitem[Gut78]{gut::}
I.~Gutman.
\newblock The energy of a graph.
\newblock {\em Ber. Math.-Statist. Sekt. Forsch. Graz}, (100-105):Ber. No. 103,
  22, 1978.
\newblock \doi{10.1007/978-3-642-59448-9_13}.
\newblock 10. Steierm\"{a}rkisches Mathematisches Symposium (Stift Rein, Graz,
  1978).

\bibitem[MPS81]{mcgpigsmi::}
O.~C. McGehee, L.~Pigno, and B.~Smith.
\newblock Hardy's inequality and the {$L\sp{1}$} norm of exponential sums.
\newblock {\em Ann. of Math. (2)}, 113(3):613--618, 1981.
\newblock \doi{10.2307/2007000}.

\bibitem[PMW{\etalchar{+}}91]{padmenwol::}
R.~Padmanabhan, N.~S. Mendelsohn, B.~Wolk, A.~B. Boghossian, D.~E. Knuth,
  J.~McCarthy, P.~Erd{\H o}s, J.~Shallit, and L.~T{\'o}th.
\newblock Elementary problems: E3432.
\newblock {\em The American Mathematical Monthly}, 98(3):263--264, 1991.
\newblock \doi{10.2307/2325034}.

\bibitem[Rud90]{rud::1}
W.~Rudin.
\newblock {\em Fourier analysis on groups}.
\newblock Wiley Classics Library. John Wiley \& Sons Inc., New York, 1990.
\newblock \doi{10.1002/9781118165621}.
\newblock Reprint of the 1962 original, A Wiley-Interscience Publication.

\bibitem[Sin38]{sin::}
J.~Singer.
\newblock A theorem in finite projective geometry and some applications to
  number theory.
\newblock {\em Trans. Amer. Math. Soc.}, 43(3):377--385, 1938.
\newblock \doi{10.2307/1990067}.

\bibitem[{Sis}18]{sis::1}
O.~{Sisask}.
\newblock {Convolutions of sets with bounded VC-dimension are uniformly
  continuous}.
\newblock {\em ArXiv e-prints}, February 2018,
  \burlalt{arXiv:1802.02836}{http://arxiv.org/abs/arXiv:1802.02836}.

\bibitem[Sze21]{sze::4}
G.~Szeg\H{o}.
\newblock \"{U}ber die {L}ebesgueschen {K}onstanten bei den {F}ourierschen
  {R}eihen.
\newblock {\em Math. Z.}, 9(1-2):163--166, 1921.
\newblock \doi{10.1007/BF01378345}.

\bibitem[TV06]{taovu::}
T.~C. Tao and V.~H. Vu.
\newblock {\em Additive combinatorics}, volume 105 of {\em Cambridge Studies in
  Advanced Mathematics}.
\newblock Cambridge University Press, Cambridge, 2006.
\newblock \doi{10.1017/CBO9780511755149}.

\bibitem[TW18]{terwol::1}
C.~{Terry} and J.~{Wolf}.
\newblock Quantitative structure of stable sets in finite {A}belian groups.
\newblock {\em ArXiv e-prints}, May 2018,
  \burlalt{arXiv:1805.06847}{http://arxiv.org/abs/arXiv:1805.06847}.

\bibitem[TW19]{terwol::0}
C.~Terry and J.~Wolf.
\newblock Stable arithmetic regularity in the finite field model.
\newblock {\em Bulletin of the London Mathematical Society}, 51(1):70--88,
  2019, \burlalt{arXiv:1710.02021}{http://arxiv.org/abs/arXiv:1710.02021}.
\newblock \doi{10.1112/blms.12211}.

\bibitem[TZ12]{tenzie::}
K.~Tent and M.~Ziegler.
\newblock {\em A course in model theory}, volume~40 of {\em Lecture Notes in
  Logic}.
\newblock Association for Symbolic Logic, La Jolla, CA; Cambridge University
  Press, Cambridge, 2012.
\newblock \doi{10.1017/CBO9781139015417}.

\end{thebibliography}

\end{document}